\documentclass[11pt]{article}
\usepackage[T1]{fontenc}
\usepackage{times} 
\usepackage{mathptmx}
\usepackage{textcomp}
\usepackage{natbib}

\usepackage{multirow}%
\usepackage{amsmath,amssymb,amsfonts}%
\usepackage{amsthm}%
\usepackage{mathrsfs}%
\usepackage[title]{appendix}%
\usepackage{xcolor}%
\usepackage{textcomp}%
\usepackage{manyfoot}%
\usepackage{booktabs}%
\usepackage{algorithm}%
\usepackage{algorithmicx}%
\usepackage{algpseudocode}%
\usepackage{listings}%

\usepackage{mathtools}

\usepackage{algorithm}
\usepackage{algpseudocode}
\usepackage{rotating}
\usepackage{comment}
\usepackage{dsfont}
\newcommand{\mathbbm}{\mathbb}

\newcommand{\Eb}{\mathbbm{E}}
\newcommand{\Rb}{\mathbbm{R}}
\newcommand{\Ib}{\mathbbm{I}}

\newcommand{\Ac}{\mathcal{A}}

\newcommand{\Fc}{\mathcal{F}}

\newcommand{\X}{\mathcal{X}}
\newcommand{\Xc}{\mathcal{X}}

\newcommand{\Rc}{\mathcal{R}}

\newcommand{\Uc}{\mathcal{U}}
\newcommand{\V}{\mathcal{V}}

\newcommand{\Pc}{\mathcal{P}}

\newcommand{\Vc}{\mathcal{V}}

\newcommand{\Zc}{\mathcal{Z}}
\renewcommand{\P}{\mathcal{P}}

\newcommand{\argmin}{\mathop{\text{argmin}}}

\renewcommand{\Omega}{\varOmega}

\newcommand{\1}{\mathds{1}}

\newcommand{\orgvec}{\vec}
\newtheorem{theorem}{Theorem}[section]
\newtheorem{lemma}[theorem]{Lemma}
\newtheorem{remark}[theorem]{Remark}

\newcommand{\eqdef}{\mathrel{\overset{\raisebox{-0.02em}{$\scriptstyle\vartriangle$}}{\,=\,}}}
\newcommand{\supp}{\mathop{\textup{supp}}}
\newcommand{\range}{\mathop{\textrm{range}}}
\DeclareMathOperator{\Proj}{\textrm Proj}

\newenvironment{tightitemize}{%
    \list{{\textup{$\bullet$}}}{\settowidth\labelwidth{{\textup{\qquad}}}
    \leftmargin\labelwidth \advance\leftmargin\labelsep
    \parsep 0pt plus 1pt minus 1pt \topsep 3pt \itemsep 3pt
    }}{\endlist}

   \makeatletter
\newcommand*\rel@kern[1]{\kern#1\dimexpr\macc@kerna}
\newcommand*\widebar[1]{%
  \begingroup
  \def\mathaccent##1##2{%
    \rel@kern{0.8}%
    \overline{\rel@kern{-0.8}\macc@nucleus\rel@kern{0.2}}%
    \rel@kern{-0.2}%
  }%
  \macc@depth\@ne
  \let\math@bgroup\@empty \let\math@egroup\macc@set@skewchar
  \mathsurround\z@ \frozen@everymath{\mathgroup\macc@group\relax}%
  \macc@set@skewchar\relax
  \let\mathaccentV\macc@nested@a
  \macc@nested@a\relax111{#1}%
  \endgroup
}
\makeatother

\title{Risk-Averse Control of Markov Systems with  Value Function Learning}

\author{Andrzej Ruszczy\'nski  and Shangzhe Yang\footnote{Department of Management Science and Information Systems, Rutgers University, email: rusz@rutgers.edu; shangzhe.yang@rutgers.edu}
}

\date{December 2023}

\begin{document}


\maketitle

\begin{abstract}We consider a control problem for a finite-state Markov system whose performance is evaluated by a coherent Markov risk measure. For each policy, the risk of a state is approximated by a function of its features, thus leading to a lower-dimensional policy evaluation problem, which involves non-differentiable stochastic operators. We introduce mini-batch transition risk mappings, which are particularly suited to our approach, and we use them to derive a robust learning algorithm for Markov policy evaluation. Finally, we discuss structured policy improvement in the feature-based risk-averse setting. The considerations are illustrated with an underwater robot navigation problem in which several waypoints must be visited and the observation results must be reported from selected transmission locations. We identify the relevant features, we test the simulation-based learning method, and we optimize a structured policy in a hyperspace containing all problems with the same number of relevant points.
\end{abstract}

\section{Introduction}

We consider a Markov decision process (MDP) with a finite state space $\Xc=\{1,\dots,n\}$, which may be very large,
the control space $\Uc$  (which is finite as well),  the  {feasible control set}  $U: \Xc\rightrightarrows \Uc$,
and the {controlled transition probability matrix} $P_{ij}(u)$, $i,j\in \Xc$, $u\in U(i)$; its {$i$-th} row, $P_i(u)$, is the distribution of the next state if the current state is $i$ and control is $u$.
 We use $\pi_t$ to denote the \emph{decision rule} to choose the control $u_t$ at time $t=0,1,\dots$,
and $\varPi=\{\pi_0,\pi_1,\dots\}$ is the \emph{policy}. In general, $\pi_t$ may be a function of $(i_0,i_1,\dots,i_t)$, the states visited
at times $0,1,\dots,t$, and produce a probability distribution on $U(i_t)$, but we shall be mainly concerned with \emph{stationary deterministic Markov policies},
in which $u_t = \pi(i_t)$ with a stationary { (time-invariant)} decision rule~$\pi$.

For any stationary deterministic Markov policy $\varPi=\{\pi,\pi,\dots\}$, and any initial state $i_0$, the sequence of states $\{i_t\}_{t=0,1,\dots}$ is a Markov chain, with the
transition probability matrix $P^\pi$ having entries $P^\pi_{ij} = P_{ij}(\pi(i))$. At each time $t=0,1,\dots$, if the state is $i_t$ and the control is $u_t$, a cost $c(i_t,u_t)$ is incurred, where $c:\Xc\times\Uc\to\Rb$. Thus, under a Markov policy $\varPi$, the resulting sequence of costs is
$c^\pi(i_t) = c(i_t,\pi_t(i_t))$, $t=0,1,\dots$.
{In standard formulations (see, \emph{e.g.},
\cite{Puterman1994,bertsekas2017dynamic}), the objective is to find a control policy that minimizes or maximizes the expected (discounted) sum or the expected average of stage-wise costs {or rewards} over a finite or infinite horizon.

Our goal is to use a dynamic measure of risk to evaluate the MDP's performance. The general theory of dynamic risk is discussed, \emph{inter alia}, in
\cite{Scandolo:2003,CDK:2006,RuszczynskiShapiro2006b,ADEHK:2007,PflRom:07,shapiro2021lectures}
 and the references therein. Our approach uses Markov dynamic risk measures introduced in \cite{Ruszczynski2010Markov}, and further developed in \cite{shen2013risk,lin2013dynamic,CavusRuszczynski2014a,CavusRuszczynski2014b,fan2018risk,fan2022process,bauerle2022markov}. In this theory, one obtains risk-averse versions of the Bellman equation, which are
 difficult to solve because of the nonlinear and nonsmooth nature of the risk measures and because of the size of the state space.
 To deal with this difficulty,
we are using value function approximation, in the spirit of \cite{sutton1988learning,bradtke1996linear,tsitsiklis1997,Sutton1998,melo2007q,yang2020reinforcement,jin2020provably}, and many references therein. However, because
of the nonlinear dependence of risk on the probability measure, these approaches are not directly applicable in our case.

Several works introduce models of risk into reinforcement learning: exponential utility functions \citep{Borkar2001, Borkar2002,basu2008learning,fei2022cascaded} and mean-variance models
\citep{Tamar2012, Prashanth2014}. A few later studies propose heuristic approaches involving specific coherent risk measures, such as
CVaR in the objective or constraints \citep{Chow2014,Chow2015a,ma2018risk}.
Generative model value iteration with coherent measures was analyzed by \cite{yu2018approximate}. Risk-aware Q-learning with Markov risk measures is considered by \cite{huang2017risk}.
A risk-averse policy gradient method was considered in
\cite{ma2017risk}. All these methods apply to problems with a small number of state-action pairs allowing exhaustive experimentation.

Value function approximations in the context of distributionally robust MDPs were considered by \cite{Tamar2014}.
Ref. \cite{Tamar2017} studies the policy gradient approach for Markov risk measures and use it in an actor-critic type algorithm. Both approaches are heuristic. Policy evaluation with linear architecture and Markov risk measures by a method of temporal differences was analyzed by \cite{kose2021risk}, and asymptotic convergence was proved.  The recent work of \cite{yin2022near} uses variance to control the value iteration procedure.
Recently, \cite{lamrisk} proposed a version
of a risk-aware reinforcement learning method with coherent risk measures and function approximation. However, at each iteration, it uses extensive experimentation (double sampling)  to statistically estimate the risk with high accuracy and high probability at each state-control pair.

In section \ref{s:markov-risk}, we briefly outline the main results of the theory of Markov risk measures and the resulting dynamic programming techniques.
In section \ref{s:mini-batch}, we introduce a new class of recursive Markov risk measures based on mini-batch transition risk mappings and study their properties. In section \ref{s:policy-eval}, we use them to develop a policy evaluation
method by simulation. We also propose a parametric policy improvement scheme based on the evaluations. Finally, in section \ref{s:robot}, we discuss the applications of these methods to the challenging problem of underwater robot navigation.

\section{Markov Risk Measures}
\label{s:markov-risk}

Consider a finite-horizon setting first.
Suppose the actions are generated by a deterministic policy $\varPi$.
A \emph{dynamic risk measure} evaluates the sequence of the random costs $Z_t=c(i_t,u_t)$, $t=0,1,\dots,T-1$ and $Z_T=c_T(i_T)$ in a risk-aware fashion. We denote by $\Zc_t$ the space
of all real functions of the history $\{i_0,\dots,i_t\}$.
 Because of the need to evaluate future costs at any period,  a dynamic risk measure is a collection of
 \emph{conditional risk measures} $\rho_{t,T}:\Zc_t \times \dots \times \Zc_T\to\Zc_t$, $t=0,1,\dots,T$, for which  we postulate three properties:

 \begin{tightitemize}
 \item[\textbf{Normalization:}] $\rho_{t,T}(0,\dots,0)=0$;
\item[\textbf{Monotonicity:}] If $(Z_{t},\dots,Z_T) \leq (W_{t},\dots,W_T)$, then $\rho_{t,T}(Z_{t},\dots,Z_T) \le \rho_{t,T}(W_{t},\dots,W_T)$;
\item[\textbf{Translation:}] $\rho_{t,T}(Z_{t},\dots,Z_T) = Z_t + \rho_{t,T}(0,Z_{t},\dots,Z_T)$.
\end{tightitemize}
Here and below, all inequalities are understood component-wise.

 Fundamental for such a nonlinear dynamic cost evaluation is \emph{time consistency},
 discussed in various forms by \cite{CDK:2006,ADEHK:2007,Ruszczynski2010Markov,cheridito2011composition}:
\emph{A dynamic risk measure is time consistent if  for every $t=0,1,\dots,T-1$, if $Z_t=W_t$ and $\rho_{t+1,T}(Z_{t+1},\dots,Z_T) \le \rho_{t+1,T}(W_{t+1},\dots,W_T)$, then}
$\rho_{t,T}(Z_{t},\dots,Z_T) \le \rho_{t,T}(W_{t},\dots,W_T)$.
As proved by \cite{Ruszczynski2010Markov}, such measures, under the conditions specified above,
must have the recursive form:
\[
\rho_{t,T}(Z_{t:T}) = Z_t + \rho_t\Big(Z_{t+1}+ \rho_{t+1}\big(Z_{t+2}+ \dots + \rho_{T-1}(Z_T)\cdots\big)\Big),
\]
 where each $\rho_t(\cdot)$ is a one-step conditional risk measure. This formula, generalizing the tower property of conditional expectations, is pertinent to our approach.

 \textit{Markov risk measures} evaluate the risk-adjusted value of future costs $(Z_{t},\dots,Z_T)$ in a Markov system with a Markov control policy
$u_t = \pi_t(i_t)$, $t=0,1,\dots,T-1$,
in such a way that the value of the future cost sequence is a function of the current state:
$\rho_{t,T}(Z_{t},\dots,Z_T) = v_{t,T}^\pi(i_t)$, where $v_{t,T}^\pi \in \Vc =\Rb^{n}$.
 This, combined with the properties specified above, implies 
 that \emph{transition risk mappings}
 $\sigma_{t,i}:  \Pc(\Xc)\times \Vc\to \Rb$, $t=0,1,\dots,T-1$, exist such that the value of each state
 for the Markovian policy $\pi$ can be evaluated by the following procedure:
 \begin{equation}
\label{DP-risk-finite}
\begin{aligned}
v_{t,T}^\pi(i) &=  c(i,\pi_t(i)) +  \sigma_{t,i}\big(P_i^{\pi_t}, v^\pi_{t+1,T}\big),
\quad  i\in \Xc, \quad t=0,1,\dots,T-1;\\
v_{T,T}^\pi(i) &= c_T(i), \quad  i \in \Xc;
\end{aligned}
\end{equation}
see \citep{Ruszczynski2010Markov,fan2022process} for the detailed derivation.
The first argument of $\sigma_{t,i}(\cdot,\cdot)$ is a probability measure on $\Xc$
(an element of the simplex in $\Rb^n$). The space of these measures is denoted by
$\Pc(\Xc)$. We use the subscript $(t,T)$ for the value function because
we shall increase $T$ to infinity in due course.

One may remark here that the risk-neutral model is a special case of \eqref{DP-risk-finite}, with the bilinear
$\sigma_{t,i}\big(P^\pi_i, v^\pi_{t+1,T}\big) =  P^\pi_{i} v^\pi_{t+1,T}$,
but we are interested in mappings that depend on $P^\pi_i$ and on $v^\pi_{t+1,T}$ in a nonlinear way.

In the infinite-horizon case, we consider the sequence of discounted costs $Z_t=\alpha^{t-1} c^\pi(i_t)$, $t=0,1,\dots$,
with some discount factor $\alpha\in (0,1)$. We also assume that
the transition risk mappings are stationary (do not depend on $t$)
and satisfy the axioms of a \emph{coherent measure of risk} \cite{ArtznerDelbaenEberEtAl1999}:
\begin{tightitemize}
\item[\textbf{Convexity:}] $\sigma_i(P_i,\lambda v + (1-\lambda) w) \leq  \lambda \sigma_i(P_i,v) + (1-\lambda) \sigma_i(P_i,w)$,  $\forall\, \lambda \in [0,1]$, $\forall\, v,w\in \Vc$;
\item[\textbf{Monotonicity:}] If $v \leq w$ (componentwise) then $\sigma_i(P_i,v) \leq \sigma_i(P_i,w)$;
\item[\textbf{Translation equivariance:}] $\sigma_i(P_i,v + \beta\1) = \sigma_i(P_i,v) + \beta$, for all $\beta \in \Rb$ (here, $\1$ is a vector of 1's);
\item[\textbf{Positive homogeneity:}] $\sigma_i(P_i,\alpha v) = \alpha \sigma_i(P_i,v)$, for all \mbox{$\alpha \geq 0$}.
\end{tightitemize}
Then, as demonstrated in \cite{Ruszczynski2010Markov}, the {infinite-horizon discounted risk measure},
\[
v^\pi(i) \eqdef \lim_{T\to\infty}v^\pi_{0,T}(i) ,\quad  i\in \Xc,
\]
is well-defined and satisfies the \emph{risk-averse policy evaluation equation}:
\begin{equation}
\label{DP-risk-infinite}
v^\pi(i) = c^\pi(i) + \alpha \sigma_i\big(P^\pi_i, v^\pi(\cdot)\big), \quad i\in \Xc.
\end{equation}
In the vector form, we can write
\begin{equation}
\label{policy-eval}
v^\pi = c^\pi + \alpha \orgvec{\sigma}\big(P^\pi, v^\pi\big),
\end{equation}
with the mapping $\orgvec{\sigma}\big(P^\pi, v^\pi\big) = \big\{ \sigma_i\big(P^\pi_i, v^\pi(\cdot)\big)\big\}_{i \in \Xc}$. Furthermore, the optimal value function $v^*(\cdot)$
satisfies the \emph{dynamic programming equation}
\begin{equation}
\label{DP-risk-optimal}
v^*(i) = \min_{u\in U(i)}\big\{ c(i,u) + \alpha \sigma_i\big(P_i^u, v^*(\cdot)\big)\big\}, \quad i\in \Xc,
\end{equation}
and the deterministic Markov policy defined by the minimizers above is the best
among all history-dependent policies.
The Reader is referred to \cite{Ruszczynski2010Markov,fan2022process,bauerle2022markov}
for the detailed derivation of these equations, for much more general state and control spaces.

\begin{remark}
{In MDP's in which the stage-wise costs at $t=1,\dots,T-1$ depend on the current state $i_t$, current control $u_t$, and the next state $i_{t+1}$, via a function $c:\Xc\times\Uc\times\Xc\to\Rb$,
and the terminal cost is $c_T(x_T)$, the policy evaluation equation \eqref{DP-risk-finite} takes on the form:
\[
v^\pi_{t,T}(i) =   \sigma_i\big(P^\pi_{i}, c(i,\pi_t(i),\cdot) + \alpha v^\pi_{t+1,T}(\cdot)\big),\  i\in \Xc, \  t=1,\dots,T-1;
\]
the final value is $v^\pi_{T,T}(i) = c_T(i)$, as before. Accordingly, the infinite horizon stationary Markov policy evaluation equation \eqref{DP-risk-infinite} and the dynamic programming equation \eqref{DP-risk-optimal}
are:
\begin{align*}
v^\pi(i) &=  \sigma_i\big(P^\pi_i, c(i,\pi(i),\cdot) + \alpha v^\pi(\cdot)\big), \quad i\in \Xc,\\
v^*(i) &= \min_{u\in U(i)} \sigma_i\big(P^u_i, c(i,u,\cdot) + \alpha v^*(\cdot)\big), \quad i\in \Xc,
\end{align*}
respectively.
}
\end{remark}
\begin{remark}
    \label{r:markov}
The derivations that lead to the policy evaluation equation \eqref{DP-risk-infinite} and the dynamic programming equation \eqref{DP-risk-optimal} are heavily dependent on the
assumptions that the dynamic measure of risk is time-consistent and Markovian. An alternative approach is to consider an overall risk measure of the total cost:
$\varrho\big(\sum_{t=0}^\infty \alpha^{t} c(i_t,u_t)\big)$. For some specific measures of risk $\varrho(\cdot)$, such as the Average Value at Risk, dynamic programming relations may be derived with the use of additional state variables; see \cite{bauerle2011markov,uugurlu2017controlled,ding2022sequential}.
\end{remark}

Eq. \eqref{DP-risk-optimal} can be solved, in principle, by the risk-averse versions of the value iteration method or the policy iteration method, which we quickly review below.
After evaluating a policy $\pi^k$ by \eqref{DP-risk-infinite}, we use the right-hand side of  \eqref{DP-risk-optimal} to find the next policy:
\begin{equation}
\label{policy-iteration-general}
\pi^{k+1}(i) = \argmin_{u\in U(i)}\big\{ c(i,u) + \alpha \sigma_i\big(P_i^u, v^{\pi^k}(\cdot)\big)\big\}, \quad i\in \Xc;
\end{equation}
the method stops when $\pi^k(i)$ is an equally good solution of \eqref{policy-iteration-general} for all $i$, and thus $v^{\pi^k}$ satisfies \eqref{DP-risk-optimal}.
The main difficulty is the implementation of this method when the state space is very large.


In addition to the
axioms of a coherent measure of risk, we shall impose additional conditions on the transition risk mappings involving their dependence on the probability measure.
For two pairs  $(P,v)$ and $(Q,w)$ the notation $(P,v){\sim} (Q,w)$ means that $P\{v \le \eta\} = Q\{ w \le \eta\}$ for all $\eta\in \Rb$ (both models have identical distribution functions). We postulate two natural properties of a transition risk mapping.
\begin{tightitemize}
\item[\textbf{Law Invariance:}] If
	 \mbox{$(P_i,v) {\sim} (Q_i,w)$ then $\sigma_i(P_i,v) = \sigma_i(Q_i,w)$};
\item[\textbf{Support Property:}] $\sigma_i(P_i,v) = \sigma_i(P_i,\1_{\supp(P_i)}v)$.
\end{tightitemize}
The symbol $\1_{\supp(P_i)}$ denotes the characteristic function of the support set of $P_i$.
The support property means that only these values $v(j)$ matter, for which a transition to state $j$ is possible.
It is automatic for the expected value but needs to be
required for general operators.
An
example is the \emph{mean--semideviation mapping} \cite{ogryczak1999stochastic}:
\begin{equation}
\label{msd-form}
\sigma_i(P_i,v) = \sum_{j\in \Xc}P_{ij}v(j)  + c \sum_{j\in \Xc}P_{ij}\Big[ v(j) -  \sum_{k\in \Xc} P_{ik}v(k)\Big]_+,
\end{equation}
with $c\in [0,1]$. Another example is the \emph{Average Value at Risk} (\cite{OgryczakRuszczynski2002,RockafellarUryasev2002}):
\begin{equation}
\label{AVaR-form}
\sigma_i(P_i,v) = \min_{\eta\in\Rb} \big\{ \eta + \frac{1}{\alpha} \sum_{j\in \Xc}P_{ij} \max (0, v(j) - \eta) \big\} , \quad \alpha \in (0,1].
\end{equation}
Yet another example, rarely used in the risk measure theory and practice, due to its conservative nature, but very relevant for us, is the \emph{worst-case mapping}:
\begin{equation}
\label{sup-form}
\sigma_i(P_i,v) = \max\big\{ v(j): P_{ij}>0,\ j\in \Xc\big\} .
\end{equation}
All examples above are coherent and law-invariant risk mappings having the support property.

The main difficulty associated with these and other mappings derived from coherent measures of risk is their statistical estimation.

\section{Mini-Batch Transition Risk Mappings}
\label{s:mini-batch}

We now propose a class of transition risk mappings that are more amenable to statistical estimation.

Suppose $\sigma_i: \Pc(\Xc)\times \Vc \to \Rb$ is a transition risk mapping. 
 If we draw a sample
$j^{1:N} = (j^1,\dots,j^N)$, with $N$ independent components distributed according to $P_i$ in $\Xc$, we obtain a random empirical measure
on $\Xc$:
$P^{(N)}(j^{1:N}) = \frac{1}{N}\sum_{m=1}^N \delta_{j^m}$, where $\delta_j$ represents the unit mass concentrated at $j$.
It is a $\Pc(\Xc)$-valued random variable on the product space $\Xc^N$ with the product measure $(P_i)^N$.
Using it as the argument of $\sigma_i$,
we obtain a random transition risk mapping $\sigma_i\big(P^{(N)}(j^{1:N}), v\big)$.
Finally, we define
\begin{equation}
\label{mini-batch}
\sigma_i^{(N)}(P_i,v) = \Eb_{j^{1:N}\sim (P_i)^N}\big\{ \sigma_i\big(P^{(N)}(j^{1:N}), v\big)\big\}.
\end{equation}
We call it the\emph{ mini-batch transition risk mapping}.

A simple example is the mini-batch transition risk mapping derived from \eqref{sup-form}:
\begin{equation}
\label{batch-sup}
\sigma_i^{(N)}(P_i,v) = \Eb_{j^{1:N}\sim (P_i)^N} \Big[ \max_{1 \le m \le N} v(j^m) \Big].
\end{equation}

We can also derive a mini-batch version of \eqref{msd-form} or \eqref{AVaR-form}. For example, the mini-batch Average Value at Risk has the following form:
\begin{equation}
\label{batch-AVaR}
\sigma_i^{(N)}(P_i,v) = \Eb_{j^{1:N}\sim (P_i)^N} \Big[ \min_{\eta\in\Rb} \big\{ \eta + \frac{1}{\alpha N}\sum_{m\in \Xc}\max (0, v(j^m) - \eta) \big\} \Big].
\end{equation}

It is evident that we may mix the mini-batch risk mappings
with the expected value, obtaining
\begin{equation}
\label{E-mix-c}
\sigma_i(P_i,v) =  (1-c)\Eb_{j^{1:N}\sim (P_i)^N}\bigg[ \frac{1}{N} \sum_{m=1}^N v(j^m) \bigg] + c \sigma_i^{(N)}(P_i,v), \quad c\in [0,1].
\end{equation}

The following observation proves that the mini-batch versions inherit the properties of the ``base'' risk measures.

\begin{lemma}\label{l:mini-batch}
If the transition risk mapping $\sigma_i(\cdot,\cdot)$ is convex (monotonic, translation equivariant, positively homogeneous, has the support property, or is law invariant)
then the mini-batch transition risk mapping $\sigma_i^{(N)}(\cdot,\cdot)$ has the corresponding property as well.
\end{lemma}
\begin{proof}
The first five properties are evident; only the law invariance requires proof. Consider two pairs:
$\big(\frac{1}{N}\sum_{m=1}^N \delta_{j^m},v\big)$ and $\big(\frac{1}{N}\sum_{m=1}^N \delta_{k^m},w\big)$.
If, for some permutation $\lambda$ of $\{1,\dots,N\}$,
\[
\big(v(j^1),\dots,v(j^N)\big) = \big(w(k^{\lambda(1)}),\dots,w(k^{\lambda(N)})\big),
\]
then the two pairs have identical distribution functions. By the law invariance of $\sigma_i(\cdot,\cdot)$, we have\break
$\sigma_i\big(\frac{1}{N}\sum_{m=1}^N \delta_{j^m},v\big) = \sigma_i\big(\frac{1}{N}\sum_{m=1}^N \delta_{k^m},w\big)$.
Thus a measurable function $\varPsi:\Rb^N\to \Rb$ exists, such that
\[
\sigma_i( P_i^{(N)}, v)=\varPsi\big(v(j^1),\dots, v(j^N)\big),
\]
 and for every permutation $\lambda$ of $\{1,\dots,N\}$,
$\varPsi\big(v(j^1),\dots, v(j^N)\big)=\varPsi\big(v(j^{\lambda(1)}),\dots, v(j^{\lambda(N)})\big)$.
Now, if $(P_i,v) \sim(Q_i,w)$, then for $j^{1:N}\sim P_i^N$ and $k^{1:N}\sim Q_i^N$, the vectors $\big(v(j^1),\dots,v(j^N)\big)$ and
$\big(w(k^1),\dots,w(k^N)\big)$ have the same distribution. Therefore,
\begin{align*}
\sigma_i^{(N)}(P_i,v) &= \Eb_{j^{1:N}\sim P_i^N}\big\{ \varPsi\big(v(j^1),\dots, v(j^N)\big)\big\}\\
 &= \Eb_{k^{1:N}\sim Q_i^N}\big\{ \varPsi\big(w(k^1),\dots, w(k^N)\big)\big\}
= \sigma_i^{(N)}(Q_i,w),
\end{align*}
which verifies the law invariance.
\end{proof}

\section{Policy Evaluation with Risk Approximation}
\label{s:policy-eval}

If the state space $\Xc$ is very large, it is impossible to tabulate the function $v^\pi(\cdot)$. Instead, we use an approximation
\[
  v(i) \approx \tilde{v}(i) \triangleq  f(\varPhi_i; \theta), \quad i \in \Xc,
  \]
of $M\ll n$ \emph{features} $\varPhi_i = \big[\varPhi_{i,1},\dots ,\varPhi_{i,M}\big]$, $i\in \Xc$, and an identical learning model $f:\Rb^M \times \Rb^d \to \Rb$ with parameters
$\theta\in \Rb^d$, where $d \ll n$.

Compactly,
$\tilde{v} = F(\varPhi; \theta)$,
with the diagonal structure: $F_i(\varPhi; \theta) = f(\varPhi_i; \theta)$, $i\in \Xc$.
An important special case is the linear model:
\begin{equation}
\label{feature-linear}
\tilde{v} =  \varPhi \theta,
\end{equation}
in which $M=d$. Thus, $\tilde{v}(i) =  \varPhi_i \theta$, $i\in \Xc$.

Such an approach is widespread in the reinforcement learning literature.
 The linear value function approximation approaches to MDPs have a long history; see
\cite{bradtke1996linear,sutton1988learning,tsitsiklis1997analysis,Sutton1998,melo2007q}
and many references therein. The concept of a linear (mixture) MDP \cite{bradtke1996linear,melo2007q}, originating from the theory of linear bandits \citep{bubeck2012regret,lattimore2020bandit},  is pertinent in this setting.
It has been recently used by
\cite{ayoub2020model,yang2020reinforcement,jin2020provably,modi2020sample} to develop complexity bounds for expected value RL algorithms.
The generalization to bilinear models by \cite{du2021bilinear} extends the applicability of this class. The state aggregation or lumping approaches of \cite{dong2019provably,katehakis2012successive} are also related to this model.
Our intention is to extend the value approximation approach to the risk-averse case.

\subsection{Abstract Policy Evaluation}

Suppose the initial state $i_1$ is random, with probability distribution $P_0$. If the Markov chain created by the policy $\pi$ is an unichain,
 we denote by $q^\pi$  its stationary distribution. If the chain is absorbing, as in our application in section \ref{s:robot}, we
 consider its restarted version, with the new transition probabilities $P_0$ at each absorbing state. If we exclude the periodic case,
 it becomes an unichain again.

As in the expected value case, we need to project the right-hand side of the risk-averse policy evaluation equation  \eqref{policy-eval} on the set in which the left-hand side lives. Therefore, we define the operator:
\begin{equation}
\label{d_op}
D^\pi(v) = \Proj_\Rc^\pi\big(c^\pi+{\alpha}  \orgvec{\sigma}(P^\pi, v)\big),\quad v\in \Vc,
\end{equation}
where $\Rc = \range(F(\varPhi;\cdot))$,
$\Proj_\Rc^\pi(w) = \argmin_{v\in \Rc } {\|v-w\|}_{q^\pi}^2$, and
the norm $\|v\|^2_q = \sum_{i\in \Xc} q_i v(i)^2$. The operator
$\Proj_\Rc^\pi(\cdot)$ is well-defined if the set $\Rc$ is convex and closed, which we shall assume from now on.

The projected risk-averse policy evaluation equation has the form:
\begin{equation}
\label{PE-proj}
\tilde{v}^\pi = D^\pi(\tilde{v}^\pi).
\end{equation}

To establish relevant properties of \eqref{d_op}, we use the dual representation of a coherent measure of risk \cite{RuszczynskiShapiro2006a},
which in our case can be stated as follows. For every $i \in \X$, a convex, closed and bounded
set $A_i(P_i) \subset \P(\X)$ exists, such that
$\sigma_i(P_i, v) = \max_{\mu_i \in A_i(P_i)}  \mu_i v$, $v \in \V$.
\begin{lemma}
\label{l:ne}
The operator $D^\pi$ is Lipschitz continuous in the $q^\pi$-norm with the modulus $\alpha\sqrt{1+\varkappa}$, where
\[
\varkappa =  \max \bigg\{ \frac{|\mu_{ij}- p_{ij}|}{p_{ij}}: \mu_i \in A_i(P_i), p_{ij} > 0, i, j \in \X \bigg\},
\]
is the distortion coefficient.
\end{lemma}
The proof follows the proof of \cite[Lem. 1]{kose2021risk}, using the fact that $\Proj_\Rc^\pi(\cdot)$ is a nonexpansive mapping in
the norm $\|\cdot\|_{q^\pi}$.




To evaluate the policy by equation \eqref{PE-proj}, we can thus (theoretically) apply a fixed point algorithm.  We define the minimizing objective with respect to a reference point $\Bar{\theta}$:
\begin{equation}
\label{L-exact}
\begin{aligned}
L_{\Bar{\theta}}(\theta) &= \big\|F(\varPhi;\theta) - c^\pi - \alpha\orgvec{\sigma}\big(P^\pi,F(\varPhi;\Bar{\theta})\big)\big\|_q^2 \\
&= \Eb_{i\sim q^\pi} \Big\{\big[f(\varPhi_i;\theta) - c_i^\pi - \alpha \sigma_i\big(P_i^\pi,F(\varPhi;\Bar{\theta})\big)\big]^2 \Big\},
\end{aligned}
\end{equation}
and we iterate
\begin{equation}
\label{value-iter}
\theta^{\ell+1} = \argmin L_{\theta^\ell}(\theta), \quad \ell=1,2,\dots.
\end{equation}
\begin{theorem}
\label{t:value-iter-conv}
If the set $\Rc$ is convex and closed and $\alpha\sqrt{1+\varkappa} < 1$, then the sequence of functions $\tilde{v}^\ell= F(\varPhi;\theta^\ell)$, $\ell=1,2\dots$,
is convergent linearly  to the unique solution of the equation \eqref{PE-proj}.
\end{theorem}
\begin{proof}
In the function space $\Vc$, the algorithm \eqref{value-iter} has the form $\tilde{v}^{\ell+1} = D^\pi(\tilde{v}^\ell)$, $\ell=1,2,\dots$. Its
linear convergence follows from the contraction property of the operator $D^\pi$ established in Lemma~\ref{l:ne}.
\end{proof}

 Evidently, if $F(\varPhi\,;\,\cdot\,)$ is an injection, then the sequence $\{\theta^\ell\}_{\ell\ge 1}$ is convergent as well.

 However, the practical implementation of the abstract scheme \eqref{value-iter} is extremely difficult. The objective function \eqref{L-exact}
 is the expected value with respect to the stationary distribution $q^\pi$. Furthermore, the values of the stage-wise transition risk mapping
 ${\sigma}\big(P_i^\pi,F(\varPhi;\Bar{\theta})\big)$ are not observed,
 and their estimation would require extensive sampling from the conditional distribution. All these difficulties make the scheme \eqref{value-iter}
 difficult to implement in its general form.

 \subsection{The Risk-Averse Method of Temporal Differences}

 Another possibility, well-developed in the expected value case, is the application of the \emph{method of temporal differences}
\cite{sutton1988learning}. It uses differences between the values of the approximate value function at successive states to improve the approximation, concurrently with the evolution of the system.
The methods of temporal differences have been proven to converge in the mean (to some limit) in \citep{Dayan1992} and almost surely by several studies,
with different degrees of generality and precision
\citep{Peng1993,Dayan1994,Tsitsiklis1994, Jaakkola1994}. Almost sure convergence of the stochastic
method with linear function approximation to a solution
of a projected dynamic programming equation was proved by \cite{tsitsiklis1997analysis}.

 The risk-averse method of temporal differences, due to \cite{kose2021risk}, is based on two
 additional assumptions:
 \begin{tightitemize}
 \item The feature model is linear, that is, it has the form \eqref{feature-linear};
 \item Unbiased statistical estimates $\widetilde{\sigma}_i\big(P_i^\pi,\varPhi\theta\big)$ of the transition risk mappings are available, satisfying
 \begin{equation}
 \label{sigma-tilde}
 \Eb\big[ \widetilde{\sigma}_i\big(P_i^\pi,\varPhi\theta\big) \,\big|\, i \big]  = {\sigma}_i\big(P_i^\pi,\varPhi\theta\big).
  \end{equation}
 \end{tightitemize}
 If we simulate a trajectory $\{i_t\}_{t\ge 0}$ of the system under the policy $\pi$, at each time step $t$, having some estimates $\theta_t$ of the
 model coefficients, we can define  the \emph{risk-averse temporal differences}:
\begin{equation}
\label{hat-dt}
d_t = \varPhi_{i_t} \theta_t - c^\pi_{i_t} - \alpha \sigma_{i_t} (P_{i_t}^\pi,\varPhi \theta_t), \quad t=0,1,2,\dots.
\end{equation}
These are the quantities under the square on the right-hand side of \eqref{L-exact} at $\theta=\bar{\theta} = \theta_t$.

The temporal differences cannot be easily computed or observed; this would require the evaluation of the risk $\sigma_{i_t} (P_{i_t}^\pi,\varPhi \theta_t)$ and thus consideration of \emph{all}
possible transitions from state $i_t$. Instead, we use their random estimates $\widetilde{\sigma}_{i_t}(P_{i_t}^\pi,\varPhi \theta_t)$, and define the observed risk-averse temporal differences,
\begin{equation}
\label{TD0-1}
\widetilde{d}_t = \varPhi_{i_t} \theta_t - c^\pi_{i_t} - \alpha \widetilde{\sigma}_{i_t} (P_{i_t}^\pi,\varPhi \theta_t), \quad t=0,1,2,\dots.
\end{equation}
This allows us to construct the {risk-averse temporal difference  method} as follows:
\begin{equation}
\label{TD0-2}
\theta_{t+1} =  \theta_t - \gamma_t \varPhi_{i_t}^\top\, \widetilde{d}_t, \quad t=0,1,2,\dots,
\end{equation}
 with appropriately chosen stepsizes $\gamma_t>0$ decreasing to 0. It has been proved in \cite{kose2021risk} that under the assumptions of Theorem \ref{t:value-iter-conv}
 and standard conditions on the sequence of stepsizes $\{\gamma_t\}_{t\ge 0}$, the sequence of functions $\{\varPhi \theta_t\}_{t \ge 0}$ is convergent a.s. to the
 solution of the projected policy evaluation equation \eqref{PE-proj}.

 The mini-batch transition risk mappings \eqref{mini-batch} provide a convenient tool to construct the statistical estimates $\widetilde{\sigma}_i\big(P_i^\pi,\varPhi\theta\big)$.
 Suppose that the original transition risk mapping is defined as \eqref{mini-batch}, with some base measure of risk $\hat{\sigma}$:
 \[
 \sigma_i(P_i^\pi ,\varPhi \theta) = \Eb_{j^{1:N}\sim (P_i^\pi )^N}\big\{ \hat{\sigma}_i\big(P^{(N)}(j^{1:N}), \varPhi \theta \big)\big\}.
 \]
  As it is defined as an expected value, the estimator
 \[
 \widetilde{\sigma}_{i_t} (P_{i_t}^\pi,\varPhi \theta_t) \triangleq \hat{\sigma}_i\big(P^{(N)}(j^{1:N}), \varPhi \theta \big)
\]
is an unbiased estimator of the transition risk which can be used in the method of temporal differences.
Its calculation involves a small number $N$ is simulations of transitions from the state $i$, and the evaluation of a
simple finite-sample risk $\hat{\sigma}_i\big(P^{(N)}(j^{1:N}), \varPhi \theta \big)$ at the simulated successor states $j^1,\dots j^N$.

 Consider the application of the mini-batch transition risk mapping \eqref{batch-sup} with $N=2$ for the risk evaluation. For the convergence of
the value iteration method \eqref{value-iter} or the method of temporal differences \eqref{TD0-2}, it is essential to estimate its distortion coefficient.
\begin{lemma}
The distortion coefficient of the  mini-batch transition risk mapping \eqref{batch-sup} with $N=2$ satisfies the inequality
\begin{equation}
\label{kappa-max}
\varkappa  \le  \max\big\{ 1  - p_{ij} : p_{ij}>0\big\}<1.
\end{equation}
\end{lemma}
\begin{proof}
Suppose, for simplicity, that the states $j$ such that $p_{ij}>0$ are $\{1,2,\dots,k\}$ and they are ordered in such a way that
$v_1 < v_2 < \dots < v_k$. At such a point $v$, the risk measure ${\sigma}(P^{(2)}(j^{1:2}),\cdot)$ is locally linear.
By construction, the distribution function of ${\sigma}(P^{(2)}(j^{1:2}),v)$ has the form
\[
P_{j^{1:2}\sim P_i^2}\big[ \max_{1 \le m \le 2} v(j^m) \le b \big]  = \Big( P_{j\sim P_i}\big[ v(j) \le b \big]\Big)^2, \quad b\in \Rb.
\]
It follows that the jumps of the distribution function have the sizes
\[
\mu_{ij} = \bigg(\sum_{s=1}^j p_{is}\bigg)^2  - \bigg(\sum_{s=1}^{j-1} p_{is}\bigg)^2 = 2 p_{ij}\sum_{s=1}^{{j}-1} p_{is} + p_{ij}^2.
\]
Therefore $p_{ij}^2 \le \mu_{ij} \le 2 p_{ij} - p_{ij}^2$. Subtracting $p_j$ and dividing by it,  we obtain the estimate \eqref{kappa-max}.
The validity of the bound at the points of nondifferentiability of ${\sigma}(P^{(2)}(j^{1:2}),\cdot)$ follows from the fact that the subdifferential
at any $v$ is the convex hull of limits of gradients at points $v'$ of differentiability, convergent to $v$.
\end{proof}

If we combine the mini-batch risk measure $\sigma^{(2)}(P_i,v)$ with the expected value, with coefficients $c$ and $1-c$, respectively,
we shall obtain $\varkappa \le c$  (recall that $c\in [0,1]$). Therefore, we can easily satisfy the contraction mapping condition
of Theorem \ref{t:value-iter-conv} by
choosing $c$ such that $\alpha \sqrt{1+c} <1$. Similar calculations can be carried out for $N>2$.

\subsection{Multi-Episodic Least-Squares Policy Evaluation}
\label{s:sample-eval}

The policy evaluation problem in \eqref{value-iter} for the linear architecture model \eqref{feature-linear} has the form:
\begin{equation}
\label{value-iter-linear}
\min_{\theta}\,{L}(\theta) \eqdef \Eb
\big[ \varPhi_{i} \theta  - c(i) - \alpha {\sigma}_{i}\big(P_{i},\varPhi{\theta^{\ell-1}}\big)\big]^2;
\end{equation}
the expectation is with respect to the stationary distribution $q^\pi$ of the state $i$. The next parameter value $\theta^{\ell+1}$ is the solution of \eqref{value-iter-linear}.

As we have already mentioned, three difficulties are associated with the problem \eqref{value-iter-linear}:
it involves the enumeration of all states, the stationary distribution  $q^\pi$ is not known, and the risk mappings $\sigma_i(P_i,\tilde{v})$ cannot be evaluated exactly. To address the first two difficulties, we sample many trajectories of the system operating under the policy $\pi$ (episodes), and we use the empirical distributions as approximations of the stationary distribution. The last difficulty is dealt with by using a mini-batch transition risk mapping
$\sigma_i^{(N)}(P_i,\tilde{v})$ for $\sigma_i(P_i,\tilde{v})$ and applying its unbiased estimates $\widetilde{\sigma}_i^{(N)}(P_i,\tilde{v})={\sigma}_i\big(P^{(N)}(j^{1:N}),\tilde{v}\big)$ at the states
encountered in the episode simulation.

Below, we outline the method with one episode per iteration, but identical
considerations apply to the multi-episodic setting.
In the iteration $\ell$ of the method, for a fixed policy $\pi$, we simulate an episode  with depth $T_\ell$; it begins with a reset state $s^{\ell}_{0}$ (sampled from some distribution $P_0$), and ends with a state $s^{\ell}_{T_\ell}$. We denote the state at time $t$ of episode $\ell$ by $s^\ell_t$. At each such state, we simulate $N$ possible successor states ${j}_{\ell,t+1}^{1:N}$ from the transition probability  $P_{s^{\ell}_{t}}$ (we write $P$ instead of $P^\pi$ for brevity).
They define the empirical measure $P^{(N)}({j}_{\ell,t+1}^{1:N})$. A randomly selected successor state from among them becomes the next trajectory state $s^{\ell}_{t+1}$;  thus, $s^{\ell}_{t+1}|s^{\ell}_{t} \sim P_{s^{\ell}_{t}}$. All sampled successors are used to calculate a random estimate of the transition risk,
\[
\widetilde{\sigma}_{s^{\ell}_{t}}^{(N)}(P_{s^{\ell}_{t}},\varPhi{\theta^{\ell-1}}) = {\sigma}_{s^{\ell}_{t}}\big(P^{(N)}({j}_{\ell,t+1}^{1:N}),\varPhi{\theta^{\ell-1}}\big).
\]
By construction,
 \begin{equation}
 \label{martingale-error}
\Eb\big[  \widetilde{\sigma}_{s^{\ell}_{t}}^{(N)}(P_{s^{\ell}_{t}},\varPhi{\theta^{\ell-1}})\,\big|\,\Fc_t\big] = {\sigma}_{s^{\ell}_{t}}^{(N)}(P_{s^{\ell}_{t}},\varPhi{\theta^{\ell-1}}),
 \end{equation}
  with $\Fc_t$ denoting the $\sigma$-algebra of all events
observed until the state $s_t^\ell$ was reached.

As an approximation of problem \eqref{value-iter-linear}, we consider the regularized sample-based Bellman error minimization:
\begin{equation}
\label{minTD_sim_target}
\min_{\theta}\,\widetilde{L}_T(\theta) \eqdef \frac{1}{T}  \sum_{t=0}^{T-1}
\Big[ \varPhi_{s^{\ell}_{t}} \theta  - c^\pi(s^{\ell}_{t}) - \alpha \widetilde{\sigma}^{(N)}_{s^{\ell}_{t}}\big(P_{s^{\ell}_{t}},\varPhi{\theta^{\ell-1}}\big)\Big]^2 + \lambda\|\theta\|^2,
\end{equation}
where $\lambda>0$ is a small regularization parameter, and $T=T_\ell$. The solution of problem \eqref{minTD_sim_target} is denoted by $\widetilde{\theta}_T^\ell$.

An essential feature of problem \eqref{minTD_sim_target} is that all its ingredients are readily available, once the episode is simulated.
For theoretical purposes, together with \eqref{minTD_sim_target}, we consider an idealized problem, with the exact values of the transition risk mappings:
\begin{equation}
\label{minTD_sim_target-ideal}
\min_{\theta}\,{L}_T(\theta) \eqdef \frac{1}{T}  \sum_{t=0}^{T-1}
\Big[ \varPhi_{s^{\ell}_{t}} \theta  - c^\pi(s^{\ell}_{t}) - \alpha {\sigma}^{(N)}_{s^{\ell}_{t}}\big(P_{s^{\ell}_{t}},\varPhi{\theta^{\ell-1}}\big)\Big]^2 + \lambda\|\theta\|^2,
\end{equation}
Its solution is denoted by $\hat{\theta}_T^\ell$.
Evidently, problem \eqref{minTD_sim_target-ideal} cannot be solved, because we do not observe the risk ${\sigma}^{(N)}_{s^{\ell}_{t}}\big(P_{s^{\ell}_{t}},\varPhi{\theta^{\ell-1}}\big)$, but it is a close approximation of the
problem \eqref{value-iter-linear} of the abstract scheme.

At first, we analyze the difference between the solutions $\widetilde{\theta}_T^\ell$ and $\hat{\theta}_T^\ell$ of problems \eqref{minTD_sim_target} and \eqref{minTD_sim_target-ideal}, respectively.
To simplify the presentation and the analysis, denote
\begin{gather*}
\varphi_t \triangleq \big(\varPhi_{s^{\ell}_{t}}\big)^\top,\\
c_t \triangleq  c^\pi(s^{\ell}_{t}),\\
\widetilde{\sigma}_t \triangleq \widetilde{\sigma}^{(N)}_{s^{\ell}_{t}}\big(P_{s^{\ell}_{t}},\varPhi{\theta^{\ell-1}}\big).
\end{gather*}
By straightforward linear algebra,
\begin{gather}
\widetilde{\theta}_T^\ell = \varLambda_T^{-1}\sum_{t=0}^{T-1}\varphi_t \big[ c_{t} + \alpha \widetilde{\sigma}_t\big], \label{thetaT}\\
\intertext{with}
\varLambda_T = \lambda\Ib + \sum_{t=0}^{T-1} \varphi_t \varphi_t^\top . \label{LambdaT}
\end{gather}
In fact, the formulas \eqref{thetaT}--\eqref{LambdaT} can be applied iteratively in the course of the simulation, for $t=1,\dots,T_\ell$, by employing the Sherman--Morrison formula
for the update of the inverse of a matrix after a rank-one modification:
\[
\varLambda_{t+1}^{-1} = \varLambda_{t}^{-1} + \frac{ \| \varLambda_{t}^{-1}\varphi_t\|^2}{1+ \varphi_t^\top \varLambda_{t}^{-1} \varphi_t },
\]
with $\varLambda_0 = \lambda\Ib$.

The exact least-squares solution of \eqref{minTD_sim_target-ideal}, if we could see the risk, would be
\begin{equation}
\label{exact-theta}
\hat{\theta}_T^\ell = \varLambda_T^{-1}\sum_{t=0}^{T-1}\varphi_t\big[ c_{t} + \alpha {\sigma}_t\big].
\end{equation}
We have the following uniform large deviation bound.
\begin{theorem}
\label{t:LD-error}
A constant $R>0$ exists such that for every $\delta>0$ with probability at least $1-\delta$, for all $T \ge 1$,
\begin{equation}
\label{LD-error}
\big\langle \widetilde{\theta}_t^\ell  - \hat{\theta}_t^\ell, \varLambda_t (\widetilde{\theta}_t^\ell  - \hat{\theta}_t^\ell) \big\rangle \le 2 \alpha^2 R^2 \ln \Big( \frac{ \big(\det(\varLambda_t)\big)^{1/2}}{\delta \lambda^{M/2}}\Big).
\end{equation}
\end{theorem}
\begin{proof}
The error between \eqref{thetaT} and \eqref{exact-theta} is given by the martingale:
\[
\widetilde{\theta}_T^\ell  - \hat{\theta}_T^\ell = \alpha \varLambda_T^{-1}\sum_{t=0}^{T-1}\varphi_t \eta_t,
\]
where $\eta_t = \widetilde{\sigma}_t - {\sigma}_t$. By \eqref{martingale-error},  $\Eb\big[\eta_t|\Fc_t] = 0$. Also, due to the finiteness of the state space,
the process $\{\eta_t\}_{t\ge 0}$ is conditionally sub-Gaussian:  a constant $R \ge 0$ exists, such that for all $t$,
\[
\Eb \big[ \textup{e}^{\lambda \eta_t} \big| \Fc_t\big] \le \textup{e}^{\lambda^2R^2/2}, \quad \forall\,\lambda \in \Rb.
\]
Define
\[
S_t = \sum_{\tau=0}^{t-1}\varphi_\tau \eta_\tau.
\]
By a concentration inequality for self-scaled martingales \cite[Thm. 1]{abbasi2011improved} (see also \cite[Thm. 14.7]{pena2009self}), for any $\delta>0$, with probability at least $1-\delta$, for all $t \ge 0$
\[
\big\langle  S_t, \varLambda_t^{-1} S_t \big\rangle  \le 2 R^2 \ln \Big( \frac{ \big(\det(\varLambda_t)\big)^{1/2}}{\lambda^{M/2}\delta}\Big).
\]
Our assertion then follows by simple algebra.
\end{proof}

The essential feature of the bound \eqref{LD-error} is that it holds uniformly in time. It is worth stressing that in the derivation of this result, we did not use anything except that $\Eb\big[\eta_t|\Fc_t] = 0$, and the finiteness of the state space. Therefore, the mechanism by which the state-control feature vectors $\{\varphi_t\}_{t \ge 0}$
are generated may be arbitrary, as long as they form an adapted sequence.

Now we use the fact that the observations are collected from a simulated path of an ergodic Markov chain. Denote by $\varphi$ the random vector of features of a state, with the
state distributed according to the stationary measure $q^\pi$.

\begin{theorem}
\label{t:Hoeffding}
Suppose the feature covariance matrix $\Eb \big[\varphi\varphi^\top\big]$ is nonsingular. Then for every $\delta>0$ and every $\varepsilon>0$
one can find $T(\delta,\varepsilon)$ such that for all $T \ge T(\delta,\varepsilon)$ with probability at least $1-\delta$ we have $\|\hat{\theta}_T^\ell - \theta^{\ell+1}\| \le \varepsilon$.
\end{theorem}
\begin{proof}
The solution of the abstract problem \eqref{value-iter-linear} has the form
\begin{equation}
\notag
\theta^{\ell+1} = \Big(\Eb \big[{\varphi} \varphi^\top\big]\Big)^{-1} \Eb\big[\varphi (c + \alpha \sigma)\big],
\end{equation}
where we write $\varphi$, $c$, and $\sigma$, for the random feature of a state, cost of a state, and transition risk at a state, when the state is distributed according to the stationary measure $q^\pi$.

The formula \eqref{exact-theta} can be rewritten as follows:
\begin{equation}
\label{exact-theta-1}
\hat{\theta}_T^\ell = \Big( \frac{\lambda}{T}\Ib + \frac{1}{T} \sum_{t=0}^{T-1} \varphi_t \varphi_t^\top\Big)^{-1} \Big( \frac{1}{T}\sum_{t=0}^{T-1}\varphi_t\big[ c_{t} + \alpha {\sigma}_t\big]\Big).
\end{equation}
A comparison of the last two displayed equations indicates that the assertion of the theorem will be true, if the empirical estimates: $\frac{1}{T} \sum_{t=0}^{T-1} \varphi_t \varphi_t^\top$ and $\frac{1}{T}\sum_{t=0}^{T-1}\varphi_t\big[ c_{t} + \alpha {\sigma}_t\big]$, will be sufficiently close to the corresponding expected values, with high probability. The term $\frac{\lambda}{T}\Ib$ is negligible
for large $T$, if $\Eb \big[\varphi\varphi^\top\big]$ is nonsingular.

For an ergodic Markov chain $\{S_t\}_{t \ge 0}$,
we  have Hoeffding-type inequalities (see \cite{glynn2002hoeffding,miasojedow2014hoeffding,moulos2020hoeffding} and the references therein): for a  function $f:\Xc \to \Rb^d$,  an absolute constant $R>0$ exists such that for all $\varepsilon_f>0$ and all $T \ge 1$
\[
P \Big\{ \Big\| \frac{1}{T} \sum_{t=0}^{T-1}  f(S_t) - \Eb[ f(S)]\Big\|_\infty > \varepsilon_f \Big\} \le 2 d \textup{e}^{ - 2 T \varepsilon_f^2/ R^2}.
\]
This particular formulation follows from \cite[Thm. 1]{moulos2020hoeffding} and the Boole--Bonferroni inequality. The application of this inequality to the two empirical estimates
in \eqref{exact-theta-1}, with proper choice of  $\varepsilon_f$ for each of them, yields the assertion of the theorem.
\end{proof}

It is clear from the proof that $T(\delta,\varepsilon)$ is of the order of  $ - \ln(\delta)/\varepsilon^2$.

It follows from Theorems \ref{t:LD-error} and \ref{t:Hoeffding} that the policy evaluation method can be implemented with high accuracy if the episodes are long enough.

\subsection{Policy Iteration}

In the policy iteration scheme, we construct two sequences $\{\pi^k\}$ and $\{\theta^k\}$ for $k = 0,1,2,..$.
Given a policy $\pi^k$, we  estimate its value by applying \eqref{minTD_sim_target}
iteratively to get the optimal $\theta^k = \theta^{\pi^k}$. Then, given the current $\theta^k$, we may try to improve the policy to get the new $\pi^{k+1}$.  In the standard approach,
the next policy $\pi^{k+1}$ follows from the previous estimated value $\theta^{k}$  via  one-step lookahead optimization:
\begin{equation}
\label{lookahead}
 \pi^{k+1}(s) = \argmin_{u\in U(s)} \{c(s,u) + \alpha \sigma_s^{(N)}(P_s(u), F(\varPhi;\theta^{k}))\}, \quad\forall s \in \X.
\end{equation}
Two practical difficulties are associated with this approach. First, $F(\varPhi,\theta^k)$ is only an approximation of the policy value $v^{\pi^k}$. Secondly, we cannot observe the accurate values of
the transition risk $\sigma_s^{(N)}(P_s(u), F(\varPhi;\theta^{k}))$; we only have access
to their statistical estimates$\widetilde{\sigma}_s^{(N)}(P_s(u), F(\varPhi;\theta^{k}))$.
Because of that, the optimization in \eqref{lookahead} may lead to inconsistent policies, resulting from
the exploitation of the inaccuracies of the model. An illustration of that is the application to be discussed in the next section, where the one-step lookahead policy turns out to be inferior to the previous one in some states.

For these reasons, we focus on a set of consistent \emph{structured policies} $\varPi(\gamma)$ controlled by a vector of parameters $\gamma \in \varGamma$.  In the simplest policy iteration scheme, our next policy $\pi^{n+1}$ follows from the previous $\theta^{k}$  via  one-step lookahead optimization over the parameter $\gamma$ for a \textit{test set} of states
$\widetilde{\Xc}$:
\begin{equation}
\label{policy-impr}
\begin{aligned}
\gamma^{k+1} &= \argmin_{\gamma \in \varGamma} \frac{1}{|\widetilde{\Xc}|}\sum_{s\in \widetilde{\Xc}}
\Big\{c\big(s,\varPi(\gamma;s)\big)
+ \alpha \widetilde{\sigma}_s^{(N)}\big(P_s(\varPi(\gamma;s)), F(\varPhi;\theta^{k})\big)\Big\}, \\
\pi^{k+1} &= \varPi(\gamma^{{k+1}}).
\end{aligned}
\end{equation}
This approach, unfortunately, is still not sufficiently stable, because the estimates  $\widetilde{\sigma}_s^{(N)}\big(P_s(\varPi(\gamma;s)),  F(\varPhi;\theta^{k})\big)$
are calculated on the basis of the value function approximations $f(\varPhi_j;\theta^{k})$ at states $j$ close to $s$, and thus close to each other; local
variations of the approximation affect the optimization model.
A multi-step lookahead policy improvement operator
is less sensitive to errors in the approximations; we describe a particular version of such a method in the next section.
We recursively run the scheme until no improvement is observed.
If the optimal policy belongs to the policy set, and if our feature-based approximations are perfect, this scheme will find the optimal policy; otherwise, a suboptimal heuristic policy will be identified.%







\section{A Robot Navigation Problem}
\label{s:robot}

\subsection{Description}

We consider the mini-batch transition risk measure and value learning in the setting of an underwater robot navigation problem. In this problem, a robot is tasked to move within a fixed connected area  $\Ac$ to visit a number of waypoints $\mathcal{W} \subset \Ac$ and to report the information $\mathcal{I} \subset \Rb_+$ collected, which can only be done at the transmission points
$\mathcal{T}\subset \Ac$. The area $\Ac$ is represented as a subset of a rectangular grid with several obstacles, but our methodology applies to other settings as well. Our goal is to control the robot in a way that minimizes the dynamic risk of the losses minus the reward for the success.

There are three kinds of actions/controls (except terminating): move the robot ($U$), collect information ($C$), and transmit information ($T$), $\Uc = U \cup C \cup T$.
The moving control set $U$ contains 8 basic directions on the grid: ${\{-1,0,1\}}^2 \setminus {\{0\}}^2$; at each location, the feasible moves are a subset of it (because of the obstacles).
Observation of a waypoint ($C$) is possible if the distance is sufficiently small, and the cost of observation is non-decreasing with respect to the distance.
The observation result is a binary random variable with either high-value information collected with probability $p$ or low-value information otherwise.  The robot will hold the information, with an equivalent reward to be received after it reaches a transmission point and reports it ($T$).

The robot may be destroyed with probability $1-\alpha$ at each step, resulting in an additional loss depending on the information collected and not yet transmitted. The additional loss function is non-decreasing with respect to the information. Since this implied discount occurs only when moving, we can simply define a control-dependent discount.

In our experiment, the state consists of 4 parts: the locations of the waypoints and the transmission points, the location of the robot, the binary variables indicating which waypoints have not been visited yet, and the information value currently carried by the robot.
That is, $\Xc = \Ac^{|\mathcal{W}| + |\mathcal{T}|} \times \Ac \times \{0,1\}^{|\mathcal{W}|} \times { \mathcal{I}}$.

The robot will not stop until it has visited all waypoints and transmitted all information, unless destroyed before. In other words, it only terminates at one of the transmission locations with all waypoints visited and zero information on hand. We define the terminating state space $\X_T = \Ac^{|\mathcal{W}| + |\mathcal{T}|} \times \mathcal{T} \times{\{0\}}^{|\mathcal{W}|} \times \{0\}$.


%
As mentioned before, we evaluate a policy by using state features.
{In our case, it is natural to require that the features have the following properties. First, they should be invariant to shape-preserving transformations, because our problem setting is not sensitive to any translation, rotation, and reflection, due to the symmetry of the moving cost $c(\cdot)$. Secondly, in approximating the value functions, for all $x \in \Xc_T$, we have $\varPhi(x;\theta) = \tilde{v}(x) = v(x) = 0$. This equation should hold for all $\theta$, therefore $\varPhi(x;\,\cdot\,) = 0$ in $\Xc_T$.

After extensive experiments, we identified several significant features listed below:
}
\begin{tightitemize}
\item The number of the non-visited waypoints;
\item The average distance among the non-visited waypoints;
\item The standard deviation of the distance among the non-visited waypoints;
\item The distance to the nearest non-visited waypoint;
\item The distance to the nearest transmission point;
\item The information collected and not transmitted yet.
\end{tightitemize}

{We define all distances in this experiment as the lengths of the shortest paths with obstacles. They can be pre-computed by the Deep First Search (DFS) algorithm.}

{The collection of functions $f(\,\cdot\,;\theta)$, $\theta \in \Rb^d$,
containing all second-order polynomials of the features, turned out to be
sufficient for approximating the value function.} Of course, the exact representation is out of the question: even a deterministic version of this problem, in its simplest setting, is equivalent to the notoriously difficult traveling salesman problem (see
\cite{junger1995traveling,cook2011traveling} and the references therein).

\vspace{1ex}
Even though the state $\Xc$ is decomposable with respect to each subproblem with specific locations of the waypoints and the transmission points $\Ac^{|\mathcal{W}| + |\mathcal{T}|}$, we combine these subproblems in a hyper-state generalized problem. If we focused on a specific subproblem only, with the locations of the waypoints and the transmission points fixed, our training procedure would fail, due to overfitting at the states visited frequently, and due to the exploitation of peculiarities of the geometry of one instance. Our value estimates at the unfamiliar states would be inaccurate, resulting in almost no policy improvement. The main advantage of considering the problem in the hyper-space is that we gain knowledge from multiple subproblems about the relevance of the geometry of the locations of the waypoints and the transmission points, and we are less likely to over-fit than in a specific subproblem. Our experimental results to be reported below indicate that a highly desirable effect is achieved: hyper-space learning improves the performance on each individual instance.

It is worth mentioning that we were not able to make the stochastic methods of temporal differences
\cite{sutton1988learning,Peng1993,Dayan1994,Jaakkola1994,tsitsiklis1997} and their risk-averse versions of \cite{kose2021risk} converge on our example. The dramatic and rare differences of transitions, due to different observation outcomes,
introduced shocks that required very small stepsizes; these, in turn,  impeded any meaningful progress.

\subsection{Simulation Process}

We focus on an example with the area  $\Ac \subset {\{0,...,9\}}^2$,  five waypoints, and two transmission points; in the hyper-space, $|\Xc| \approx 10^{20}$. A linear feature-based model is used to learn the estimated value, that is $\tilde{v} = F(\varPhi;\theta) = \varPhi \theta + \theta_0$, with $\varPhi$ containing all first- and second-order terms in the feature space. Our structured policy is a heuristic threshold policy $\varPi(\gamma\,;\,\cdot\,)$ controlled by a single parameter $\gamma$:
\[
\varPi(\gamma;s) =
\begin{cases}
    \text{go to the nearest transmission point,} & 
  \hspace{-1.5em}\text{if } I > 0 \text{ and } \min_{dW} {\ge} \gamma(\mathcal{W}, \mathcal{T}) \frac{\min_{dT}}{I};\\
    \text{go to the nearest non-visited waypoint,} & \
    \qquad \qquad \text{otherwise;}
\end{cases}
\]
where $\min_{dW}$ is the distance to the nearest non-visited waypoint,  $\min_{dT}$ is the distance to the nearest transmission point, and $I$ is the value of information carried. We allow the threshold parameter $\gamma(\mathcal{W}, \mathcal{T})$ to be different for different configurations. { Although the policy may change its destination on the way, it }is applicable to any connected search space, with an appropriate definition of the distance.

To solve the hyper-state problem  with estimated mini-batch risk measure, we uniformly sample $J$ configurations of the waypoints and the transmission points,  $\{(\mathcal{W}_1, \mathcal{T}_1),$
$\dots , (\mathcal{W}_J, \mathcal{T}_J)\}$. For each configuration, we simulate the trajectory samples $\{\xi_1, ..., \xi_J\}$ and average their estimated objectives in \eqref{minTD_sim_target}.

Each trajectory sample $\xi_i$ contains $M$
episodes sampled from different random initial states;
 the number of episodes is equal to the number of non-visited waypoints.
 The length of episode $m$ is denoted by $T_m$.
Also, we assume that no destruction happens during the simulation; the destruction probability is accounted for by the discount.

The total sample-based Bellman error for the robot navigation problem can be written as:
\[
\widetilde{L}(\theta) = \frac{1}{T} \sum_{j=1}^J \sum_{m=1}^M \sum_{t=1}^{T_m}\big[\varPhi_{s_{j,m,t}}\theta - c(s_{j,m,t}) - \alpha \widetilde{\sigma}^{(N)}_{s_{j,m,t}}\big(P_{s_{j,m,t}},\varPhi {\theta}\big)\big]^2,
\]
where $T$ is the total number of states visited.
Once we have estimated  $\theta$  by approximately minimizing $\widetilde{L}(\cdot)$, we may use the estimated
values $\widetilde{v}$ to improve the policy.

The one-step lookahead policy \eqref{lookahead} fails miserably on this problem because it depends on the local variations of the value function approximation about each state, which are subject to errors. Also, an improvement of the parameter $\gamma$ in the parametric policy
$\varPi(\gamma\,;\,\cdot\,)$ cannot be easily derived by comparing the model values at neighboring states. For this reason, we use a
specially designed variable-depth lookahead policy improvement operation. Its main idea is to use the estimated policy values { $\widetilde{w}_\gamma(s)$}
accumulated on the way from the state $s$ to the waypoint or the transmission point to which the policy leads.

For each new configuration (that might not be one of
the configurations included in the training process), and a freshly sampled test set of states $\widetilde{\Xc}$, we solve
 the minimization problem
\begin{equation}
\label{gamma-step}
\gamma^{k+1} = \argmin_{\gamma \in \varGamma} \frac{1}{|\widetilde{\Xc}|} \sum_{s\in \widetilde{\Xc}} \widetilde{w}_\gamma(s),
\end{equation}
where the lookahead value $\widetilde{w}_\gamma(s)$ is calculated by the recursion:
\begin{equation}
\label{policy-impr-multi}
\widetilde{w}_\gamma(s) =
\begin{cases}
    \widetilde{\sigma}_s^{(N)}\big(P_s(\varPi(\gamma;s)), F(\varPhi;\theta^{n})\big) & \
    \text{if } \varPi(\gamma;s) \in C \cup T;\\
    c\big(s,\varPi(\gamma;s)\big) + \alpha \widetilde{w}_\gamma(s_{\text {next}}) & \
    \text{otherwise.}
\end{cases}
\end{equation}
Here, $s_{\text{next}}$ is the state that follows $s$ when the control $\varPi(\gamma;s)$ is used.
After that, we set $\pi^{k+1} = \varPi(\gamma^{{k+1}})$. Observe that the lookahead depth in \eqref{policy-impr-multi} depends on the current state
and on the value of $\gamma$.
This variable-depth multi-step procedure helps to reduce the impact of errors at neighboring states, thus improving performance against the one-step estimation in \eqref{policy-impr}.


\subsection{Results}

In the policy iteration scheme, we randomly sampled $J = 50$ subproblems and $M = 80$ episodes and we tried the risk-neutral expectation measure ($N = 1$) and the mini-batch risk measure \eqref{batch-sup} with $N = 2$. We begin from an initial heuristic policy with a constant $\gamma$ for all subproblems, evaluate it in the hyperspace, and then, for every new configuration, we carry out one step of the policy improvement scheme.

Figure \ref{f:path} presents two typical samples of the configurations tested; none of them was used in the training process. It also provides the average performance of the three policies and its upper semideviation (an integrated measure of risk). We find that the learned policies of the mini-batch risk measure with $N = 2$ improve from the initial heuristic policy and are closer to the optimal policies solved by the dynamic programming equation (which was still possible to solve for each individual instance). Interestingly, they frequently outperform the learned policies for the risk-neutral (expectation) measure; apparently, the conservative nature of the risk mapping neutralizes the imperfections of the feature-based approximation. The improvement occurs on \emph{all} problem instances tested, {with distinct final policy $\gamma^*(\mathcal{W}, \mathcal{T})$ determined by the locations of the waypoints and the transmission points.}

In the top configuration, for a particular sequence of observation outcomes, in the learned policy based on the risk-neutral measure, the robot follows the trajectory $\color{blue}{(7,0)} \to \color{teal}{(5,4)} \to \color{teal}{(2,4)} \to \color{teal}{(5,4)} \to \color{red}{(1,4)} \to \color{teal}{(2,2)} \to \color{red}{(1,4)} \to \color{teal}{(3,8)} \to \color{teal}{(8,7)} \to \color{red}{(6,6)}$. In the learned policy based on the risk-averse measure, the robot is performing conservatively at the end with the trajectory $ \color{blue}{(7,0)} \to \color{teal}{(5,4)} \to \color{teal}{(2,4)} \to \color{teal}{(5,4)} \to \color{red}{(1,4)} \to \color{teal}{(2,2)} \to \color{red}{(1,4)} \to \color{teal}{(3,8)} \to \color{red}{(6,6)} \to \color{teal}{(8,7)} \to \color{red}{(6,6)}$.

In the bottom configuration, again for a particular sequence of observation outcomes, in the learned policy based on the risk-neutral measure, the robot follows the
trajectory $\color{blue}{(4,2)} \to \color{teal}{(4,5)} \to \color{teal}{(2,6)} \to \color{red}{(1,1)} \to \color{teal}{(7,5)} \to \color{teal}{(9,4)} \to \color{red}{(7,2)} \to \color{teal}{(3,9)} \to \color{red}{(7,2)}$. In the learned policy based on the risk-averse measure, the robot follows the
trajectory $\color{blue}{(4,2)} \to \color{teal}{(4,5)} \to \color{red}{(7,2)} \to \color{teal}{(7,5)} \to \color{red}{(7,2)} \to \color{teal}{(9,4)} \to \color{red}{(7,2)} \to \color{teal}{(6,2)} \to \color{red}{(1,1)} \to \color{teal}{(3,9)} \to \color{red}{(7,2)}$.

Since the policies differ only at the end in the top configuration, the risk-averse learned policy does not outperform the risk-neutral one so much in the statistics. However, in the bottom configuration, it turns out that the risk-averse policy is far better than the risk-neutral one in two statistics:  the expected value measure, and the upper-semi-deviation.

\begin{figure}[t]
  \hspace{-2em}\includegraphics[width=13.5cm]{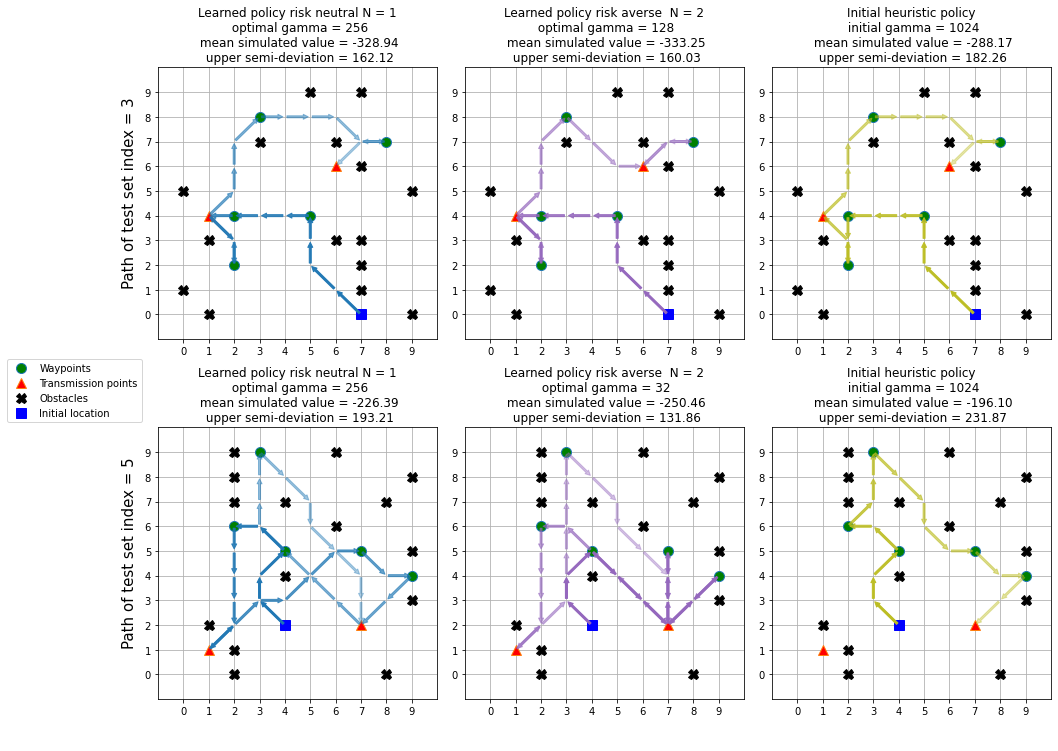}
  \caption{Robot trajectories for different policies on two random configurations (top and bottom). On the left are the trajectories of the robot using the learned risk-neutral policy. In the middle are the trajectories of the robot using the learned risk-averse policy, with the mini-batch size $N=2$. On the right are the trajectories of the robot using the initial policy of moving to the nearest relevant point (an unobserved waypoint or a transmission point). On top of each plot, we provide performance statistics obtained by multiple simulations.}
  \label{f:path}
\end{figure}


\section{Conclusion}

A number of observations follow from the research reported here:
\begin{tightitemize}
\item Markov risk measures with mini-batch transition risk mappings lead to a tractable
feature-based policy evaluation problem for which an implementable algorithm can be designed;
\item It is useful to train in a high-dimensional hyperspace, involving the problem configuration;
\item Stochastic methods of temporal difference learning do not work well on our test example, due to the highly random outcomes at some states;
\item Policy improvement within parametric policies and with multi-step look-ahead models stabilizes the
learning process;
\item Risk aversion neutralizes the effect of the imperfections of feature-based models.
\end{tightitemize}
We believe that our observations provide a preliminary insight into risk-averse learning and control with value function approximation for Markov decision problems.

\noindent



\bibliographystyle{abbrv}

\end{document}